\documentclass[11pt]{article}

\RequirePackage{amsthm,amsmath,amsfonts,amssymb}
\RequirePackage[authoryear]{natbib}
\theoremstyle{plain}
\newtheorem{thm}{Theorem}
\newtheorem{lemma}{Lemma}[section]

\newtheorem{prop}[lemma]{Proposition}

\usepackage[margin=1.25in]{geometry}
\usepackage{setspace}
\usepackage{titlesec}
\usepackage{authblk}
\usepackage{abstract}

\titleformat{\section}
{\normalfont\scshape\center}
{\thesection.}{0.5em}{}

\titleformat{\subsection}
{\normalfont\itshape}
{\thesubsection.}{0.5em}{}

\begin{document}

\renewcommand\Authfont{\scshape}
\renewcommand\Affilfont{\normalfont\small}
\setlength{\affilsep}{1em}

\title{Monotonicity of critical point in two-dimensional oriented percolation with enhancement}

\author{Célio Terra\\
	Universidade Federal do Rio de Janeiro\\
	\texttt{caugusto.terra@gmail.com}
}

\date{}

\maketitle
	
	\begin{abstract}
		In this note, we investigate Bernoulli oriented bond percolation with parameter $p$ on $\mathbb{Z}^2$. In addition to the standard edges, which are open with probability $p$, we introduce diagonal edges each open with probability $\varepsilon$. Every edge is open or closed independently of all other edges. We prove that the critical parameter for this model is strictly decreasing in $\varepsilon$. 
	\end{abstract}
	
\bigskip

	%%%%%%%%%%%%%%%%%%%%%%%%%%%%%%%%%%%%%%%%%%%%%%
	%%%% Main text entry area:
	\section{Introduction} \label{sec:introduction}

	One of the most studied aspects of percolation theory is the concept of \emph{enhancements}. Informally, an enhancement refers to a rule that introduces additional open edges to the system. A natural question arises: does the addition of those extra edges facilitate the existence of an infinite open cluster? For non-oriented percolation, it is well-established that a broad class of enhancements reduces the value of the critical point (see, for example,~\cite{AizenmanGrimmett1991} and~\cite{Grimmett1999}). For oriented bond percolation, the usual enhancement techniques are not directly applicable. Some results about monotonicity of the critical points for oriented percolation in the three-dimensional hexagonal lattice have been found~(\cite{deLimaUngarettiVares2024}). Dynamic enhancements, where the rule depends on the configuration, were examined in~\cite{AndjelRolla23} for the contact process, which is the continuous-time analogue of oriented percolation. In that context, the monotonicity of the critical point was also proven.
	
	In this note, we investigate a static enhancement in an oriented percolation model. Specifically, we consider the classic oriented Bernoulli bond percolation on $\mathbb{Z}^2$ with parameter $p$, augmented with additional edges  that are open with probability $\varepsilon$ independently of each other and of all other edges. We establish that the critical value of this model strictly decreases as $\varepsilon$ increases.
	
	As usual when dealing with oriented percolation in $\mathbb{Z}^2$, we will consider the model in a rotated version of $\mathbb{Z}^2$. Let 
	\[\Lambda:=\{(m,n)\in \mathbb{Z}^2; n \ge 0, m+n \text{ is even}\}.\] 
	We consider two sets of oriented bonds,
	\[\mathcal{B}:=\{(m,n)\rightarrow (m+1,n+1),(m,n) \rightarrow (m-1,n+1):(m,n) \in \Lambda\}\]
	and
	\[\mathcal{B}_e:=\{(m,n)\rightarrow (m,n+2):(m,n) \in \Lambda\} \cup \mathcal{B}.\]
	
	Consider two oriented graphs, $G := (\Lambda, \mathcal{B})$ and $G_e := (\Lambda, \mathcal{B}_e)$. Each bond in $\mathcal{B}$ is open with probability $p$ and closed with probability $1 - p$, while each bond in $\mathcal{B}_e \setminus \mathcal{B}$ is open with probability $\varepsilon$ and closed with probability $1 - \varepsilon$. All bonds are open or closed independently of the state of all other bonds.
	
	Let $\mathbb{P}_{p,\varepsilon}$ denote the law of the oriented bond percolation model on $G_e$, defined in terms of the above occupation variables. The notation $\mathbb{P}_p := \mathbb{P}_{p,0}$ represents the usual Bernoulli-oriented bond percolation on $G$.
	
	A coupling is introduced to construct oriented percolation processes with different parameter values $p$ and $\varepsilon$ on the same probability space. To each bond $(m,n) \rightarrow (m',n')$, associate an independent random variable $Y_{(m,n)}^{(m',n')}$ with a uniform distribution on $[0,1]$. For $\lambda \in [0,1]$, the bond $(m,n) \rightarrow (m',n')$ is defined as \emph{$\lambda$-open} if $Y_{(m,n)}^{(m',n')} \le \lambda$ and \emph{$\lambda$-closed} if $Y_{(m,n)}^{(m',n')} > \lambda$.

	A \emph{$(p,\varepsilon)$-open path} is an oriented path on $G_e$ consisting only of edges on $\mathcal{B}$ that are $p$-open and edges of $\mathcal{B}_e \setminus \mathcal{B}$ that are $\varepsilon$-open. Given $A,B \subseteq \Lambda$, we say that $A \xrightarrow{p, \varepsilon} B$ if there is a $(p,\varepsilon)$-open path connecting $A$ to $B$. We define $\{(x,n)\xrightarrow{p, \varepsilon} \infty\}:=\cap_{m>n}\{(x,n) \xrightarrow{p, \varepsilon} \Lambda \cap (\mathbb{Z} \times \{m\})\}$.
	
	The \emph{oriented percolation process with parameters $(p,\varepsilon)$ and initial condition $A \subset 2 \mathbb{Z}$}, denoted by $(N_n^A)_{n \ge 0}$ is defined as
	\[N_n ^A:= \{m \in  \mathbb{Z} : A \times \{0\} \xrightarrow{p, \varepsilon} (m,n)\}.\]
	
	If $A=2\mathbb{Z}^-:=\{m \in 2\mathbb{Z}, m \le 0\}$, we write $N^-_n$ instead of $N^A_n$. 
	
	We will also see $N^A_n$ as a random subset of $\mathbb{Z}$. Let $\mathcal{P}(\mathbb{Z})$ be the powerset of $\mathbb{Z}$ with its usual Borel $\sigma$-algebra and partial order given by inclusion. Given $A$ and $B$ two random subsets of $ \mathbb{Z}$, we say that $A$ is \emph{stochastically dominated} by $B$, denoting by $A \preccurlyeq B$, if $\mathbb{E}[f(A)] \le \mathbb{E}[f(B)]$ for every increasing function $f$.
	
	Let
	\[\theta(p, \varepsilon)=\mathbb{P}_{p,\varepsilon}((0,0)\xrightarrow{p,\varepsilon} \infty).\]
	For $\varepsilon \in [0,1]$, the \emph{critical parameter} $p_c(\varepsilon)$ is defined as
	\[p_c(\varepsilon):=\inf\{p: \theta(p,\varepsilon)>0\}.\] 
	We can ask ourselves if the extra ``help'' given by the open bonds on $\mathcal{B}_e \setminus \mathcal{B}$ lowers the critical parameter of the system. The answer to this question is yes, as shown in our main result.
	
	\begin{thm} \label{thm:monotonicityepsilon}
		For every $0 \le \varepsilon< \tilde{\varepsilon}\le 1$, $p_c(\tilde{\varepsilon})<p_c(\varepsilon)$. In particular, $p_c(\varepsilon) < p_c(0)$ for all $\varepsilon \in (0,1]$.
	\end{thm}
	
	Theorem~\ref{thm:monotonicityepsilon} will be proved on Section~\ref{section:monotonicitycriticality}. Building on arguments similar to those in~\cite{AndjelRolla23}, we show that the presence of additional open edges contributes more to the cluster's growth than simply translating the cluster one unit to the right. This leads to an increase in the speed of the cluster's right edge, thereby directly establishing both theorems.

	\section{Speed of the right edge} \label{section:rightedge}
	
	In this section, we will show that the right edge of the percolation cluster has a well defined speed, and show that this speed is zero at the critical point.
	
	From the definitions of the Section~\ref{sec:introduction} it follows that the oriented percolation process with parameters $(p,\varepsilon)$ is \emph{attractive}. This means that, for all $A,B \subseteq 2 \mathbb{Z}$,
	\[N^{A \cup B}_n \subseteq N^A_n \cup N_n^B \text{ for all }n \ge 0.\]
	
	Define, for every non-empty	 $A \subseteq \mathbb{Z}$, $r(A):= \sup A$.
	\begin{prop} \label{prop:existencevelocity} For all $p, \varepsilon$ in $[0,1]$ there is $ \alpha=\alpha(p,\varepsilon)$ such that
		\[\lim_n \frac{r(N^-_n)}{n}=\alpha\]
		almost surely.
	\end{prop}
	\begin{proof}
		It follows from attractiveness and Kingman's subadditive ergodic theorem, as shown in~\cite[Theorem VI.2.19]{Liggett05}.
	\end{proof}
	
	As the right edge can increase by $1$ at each unit of time, $\alpha(p, \varepsilon) \le 1$ a.s. A key property of $\alpha$ is that it is zero at criticality.
	\begin{prop} \label{prop:zerospeed}
		For all $\varepsilon\in [0,1]$,
		\[\alpha(p_c(\varepsilon), \varepsilon)=0.\]
	\end{prop}
	\begin{proof}
		The proof uses the arguments of~\cite[Section 9]{Dur84}. Fix $p, \varepsilon \in [0,1]$ and suppose $\alpha:=\alpha(p,\varepsilon)> 0$, the case $\alpha<0$ is analogous. Given $L \in \mathbb{N}$, define $D$ to be the parallelogram with vertices $(-0.15 \alpha L,0)$, $(-0.05 \alpha L, 0)$, $(0.95 \alpha L,1.1  L)$ and $(1.05\alpha L, 1.1  L)$, and define the event
		\begin{align*}
			E:=\left\{\begin{alignedat}{4}
				& \text{There is a $(p,\varepsilon)$-open path from }[-0.15 \alpha L, -0.05 \alpha L]\times\{0\} \\
				& \text{to }[0.95 \alpha L, 1.05 \alpha L]\times \{1.1 L\}  \text{ that do not leave } D \\
				&\text{and a $(p,\varepsilon)$-open path from }[0.05 \alpha L, 0.15 \alpha L] \times \{0\} \\
				& \text{ to }[-1.05 \alpha L, -0.95 \alpha L] \times \{1.1 L\} \text{ that do not leave }-D
			\end{alignedat}	
			\right\}.
		\end{align*}
		As $r(N^-_n)/n \rightarrow \alpha$, almost surely, $\mathbb{P}_{p, \varepsilon}(E)$ can be made arbitrarily close to  $1$ if we take $L$ large enough. This fact follows from straightforward arguments (see~\cite[Section  9]{Dur84}).
		
		For $(m,n) \in \Lambda$, define $E_{m,n}$ to be the translation of event $E$ by the vector $(0.9 \alpha Lm, Ln)$. Note that $E_{m,n}$ is independent of $E_{m',n'}$ if $(|m-m'|+|n-n'|)/2>1$. Furthermore, by translation invariance, $\mathbb{P}_{p,\varepsilon}(E_{m,n})=\mathbb{P}_{p,\varepsilon}(E)$ for every $(m,n)\in \Lambda$. 
		
		We define a collection $\mathcal{P}:=\{\eta(m,n)\}_{(m,n) \in \Lambda}$ of Bernoulli random variables such that $\eta(m,n)=1$ if $E_{m,n}$ occurs and $\eta(m,n)=0$ otherwise. We see that $\mathcal{P}$ forms a $1$-dependent site percolation model such that, if there is an infinite open cluster in $\mathcal{P}$, then there is an infinite path in the original oriented bond percolation with enhancement system. By~\cite{Liggett1997}, there is $\delta>0$ such that, if $\mathbb{P}(\eta(m,n)=1)>1-\delta$, then there is a positive probability of existing an infinite open cluster in $\mathcal{P}$. Take $L$ large enough such that $\mathbb{P}_{p,\varepsilon}(E_{m,n})>1-\delta$. Since $E_{m,n}$ depends on a finite collection of edges, we can take $p'<p$ such that, for the same $L$, we still have $\mathbb{P}_{p',\varepsilon}(E_{m,n})>1-\delta$, and thus, there will be a positive probability of existence of an infinite $(p', \varepsilon)$-open path in the oriented bond percolation with enhancement system. Thus, if $\alpha(p, \varepsilon)>0$, then $p>p_c(\varepsilon)$.
	\end{proof}

	\section{Monotonicity of critical points} \label{section:monotonicitycriticality}
	
	In this Section we prove Theorem~\ref{thm:monotonicityepsilon}. To do this, we make use of two lemmas. The first one states that adding a site to the right edge of the percolation cluster ``helps'' the process as seen from the right edge. The second one tells that this extra help can be done in such a way that it is preserved by the evolution of the process.
	
	Before stating the lemmas, we will need a definition. For $A \subset 2 \mathbb{Z}$ with $\sup A < +\infty$, we define $F(A)=\{x-r(A) : x \in A\}$. 
	\begin{lemma} \label{lemma:domination}
		For every $p$, $\varepsilon \in [0,1]$ and $n \ge 0$,
		\begin{equation} \label{eq:domination}
			F(N^{-}_n)\preccurlyeq F(N^{-}_n \cup \{r(N^{-}_n)+2\}).
		\end{equation}
	\end{lemma}
	
	\begin{lemma} \label{lemma:coupling}
		If $F(A) \preccurlyeq F(B)$ then for any values of $p$ and $\varepsilon$ the processes $(N^A_n)_{n \ge 0}$ and $(N^B_n)_{n \ge 0}$ can be coupled in such a way that $r(N^A_n)-r(A) \le r(N_n^B)-r(B)$ almost surely for all $n$.
	\end{lemma}
	
	Lemmas~\ref{lemma:domination} and~\ref{lemma:coupling} will be proved on Section~\ref{section:stochasticdomination}.
	
	\begin{proof}[Proof of Theorem~\ref{thm:monotonicityepsilon}.] Let $\tilde{\varepsilon}>\varepsilon\ge 0$, and $(N_n^{-})_{n\ge 0}$, $(\tilde{N}^-_n)_{n\ge 0}$ oriented percolation processes with parameters $(p_c(\varepsilon), \varepsilon)$ and $(p_c(\varepsilon), \tilde{\varepsilon})$, respectively, both with initial condition $2 \mathbb{Z}^-$. 
		
		Define the random time
		\begin{align*}
			\tau_1 := \min \{&n \ge 2 :  r(N^{-}_n)= r(N^{-}_{n-1})-1 = r(N^{-}_{n-2})-2; \\&	Y_{(r(N^{-}_{n-2}), n-2)}^{(r(N^{-}_{n-2}), n)} \in (\varepsilon, \tilde{\varepsilon}]\}.
		\end{align*}
		
		In other words, $\tau_1$ is the first time the right edge of the process ${(N^-_n)_{n \ge 0}}$ moves to the left two times in a row starting at some point $(x,n-2)$ and the vertical bond starting at that point is $\tilde{\varepsilon}$-open but $\varepsilon$ closed. It is clear that the $\tau_1$ is a finite stopping time for the natural filtration of the process $(N_n^-)_{n \ge 0}$. Moreover, the tail of $\tau_1$ has a geometric decay, because $\tau_1$ is stochastically dominated by a geometric random variable. 
		
		By the definition of $\tau_1$, 
		\[N^{-}_{\tau_1}\cup \{r(N^{-}_{\tau_1})+2\} \subseteq \tilde{N}^{-}_{\tau_1}.\]
		
		By Lemma~\ref{lemma:domination}, 
		\[F(N^{-}_{\tau_1}) \preccurlyeq F(N^{-}_{\tau_1}\cup \{r(N^{-}_{\tau_1})+2\})\preccurlyeq F(\tilde{N}^-_{\tau_1}).\] 
		Applying Lemma~\ref{lemma:coupling}, there is a coupling such that
		\[r(N^{-}_{n})+2 \le r(\tilde{N}^{-}_{n}) \text{ a.s.\ for every } n \ge \tau_1.\]

		We define $(N^1_n)_{n\ge \tau_1}$ and $(\tilde{N}^1_n)_{n\ge \tau_1}$ as oriented percolation processes with initial condition $\tilde{N}^{-}_{\tau_1}$ and parameters $(p_c(\varepsilon),\varepsilon)$ and $(p_c(\varepsilon), \tilde{\varepsilon})$ respectively. 
		
		We define inductively, for $k \ge 2$, the stopping time $\tau_k$ as
		\begin{align*}
			\tau_k := \min \{&n \ge \tau_{k-1}+2 :  r(N^{k-1}_n)= r(N^{k-1}_{n-1})-1 = r(N^{k-1}_{n-2})-2; \\&	Y_{(r(N^{k-1}_{n-2}), n-2)}^{(r(N^{k-1}_{n-2}), n)} \in (\varepsilon, \tilde{\varepsilon}]\}.
		\end{align*}
		\noindent We define the processes $(N^k_n)_{n\ge \tau_k}$ and $(\tilde{N}^k_n)_{n\ge \tau_k}$ as oriented percolation processes with parameters $(p_c(\varepsilon), \varepsilon)$ and $(p_c(\varepsilon), \tilde{\varepsilon})$ respectively, and initial condition $N^k_{\tau_k}=\tilde{N}^k_{\tau_k}=\tilde{N}^{k-1}_{\tau_k}$. Again by Lemmas~\ref{lemma:domination} and~\ref{lemma:coupling}, there is a coupling of $(N^k_n)_{n \ge \tau_k}$ and $(\tilde{N}^k_n)_{n \ge \tau_k}$ such that
		\[r(N^{k}_{n})+2 \le r(\tilde{N}^{k}_{n}) \text{ a.s.\ for every } n \ge \tau_{k+1}.\]
		
		Analogously, we can find a coupling such that
		\begin{equation} \label{eq:inequalityedge}
			r(N^{-}_{n})+2k \le r(\tilde{N}^{-}_{n}) \text{ a.s.\ for every } n \ge \tau_{k+1}.
		\end{equation}
		Furthermore, $\tau_{k+1}-\tau_{k}$ has the same distribution of $\tau_1$, and thus its tail has a geometric decay for all $k$. In particular, $\lim_{k} \frac{\tau_k}{k}<+\infty$. Therefore, by~\eqref{eq:inequalityedge},
		\begin{align*}
			\lim_{n} \frac{r(\tilde{N}_n^-)}{n} & \ge \lim_{k}  \frac{r(\tilde{N}_{\tau_k}^-)}{\tau_k} \\
			& \ge \lim_{k} \frac{r(N_{\tau_k}^-)}{\tau_k}+\frac{2k}{\tau_k}.
		\end{align*}
		
		By Proposition~\ref{prop:zerospeed},  
		\[\lim_{k} \frac{r(N_{\tau_k}^-)}{\tau_k}=\alpha(p_c(\varepsilon), \varepsilon)=0\]
		and thus, $\alpha(p_c(\varepsilon), \tilde{\varepsilon})>0$. Again by Proposition~\ref{prop:zerospeed}, $\alpha(p_c(\tilde{\varepsilon}), \tilde{\varepsilon})=0$. Hence, $p_c(\varepsilon)<p_c(\tilde{\varepsilon})$.
	\end{proof}
	
	\section{Stochastical domination by the enhanced process} \label{section:stochasticdomination}
	
	In this Section we prove Lemmas~\ref{lemma:domination} and~\ref{lemma:coupling}.
	\begin{proof}[Proof of Lemma~\ref{lemma:domination}.]
		We borrow the arguments of~\cite{AndjelRolla23}. Fix $n \ge 1$ and let $\Gamma$ be the rightmost open path from $2 \mathbb{Z}^- \times \{0\}$ to $\Lambda \cap (\mathbb{Z} \times \{n\})$. For any deterministic path $\gamma$ connecting $2 \mathbb{Z}^- \times \{0\}$ to $\Lambda \cap (\mathbb{Z} \times \{n\})$, the event $\{\Gamma =\gamma\}$ is determined only by the state of bonds on $\gamma$ and to the right of $\gamma$, i.e., on the set $D_{\gamma}:=\mathbb{Z}^2 \cap \{(x,k):k= 0,1,\dots, n, x \ge \gamma_k\}$. On the other hand, for every $(x,n) \in \Lambda$ with $x<\gamma_n$, the event $\{(x,n) \in N^-_n\}$ is determined only by the state of the edges on the set $E_{\gamma}:=\{(x,k): k =0,1,\dots, n, x<\gamma_k\}$. By independence of the edges, the conditional distribution of $N^-_n$ given $\{\Gamma=\gamma\}$ coincides with the trace on $\Lambda \cap (\mathbb{Z} \times \{n\})$ of the $\mathbb{P}_{p, \varepsilon}$ oriented percolation cluster of $C_{\gamma}:=(2 \mathbb{Z} \times \{0\}) \cup \{(\gamma_k,k), k=0,1,\dots, n\}$ in $E_{\gamma}$. Shifting everything by the space coordinate $\gamma_n$, i.e., substituting $C_{\gamma}$ by $C_{\tilde{\gamma}}$, with $\tilde{\gamma}_k=\gamma_k-\gamma_n$, $k=0,\dots,n$, we obtain the measure on the left-hand side of~\eqref{eq:domination}.
		
		Using the same arguments, the measure on the right-hand side of~\eqref{eq:domination}  is given by the analogous percolation cluster with $C_{\tilde{\gamma}}$ replaced by $C_{\bar{\gamma}}$, where $\bar{\gamma}_k=\tilde{\gamma_k}-2$ for $k=0,1,\dots n-1$ and $\bar{\gamma}_n={\tilde{\gamma}}_n=0$. The result follows from planarity of the graph.
	\end{proof}
	\begin{proof}[Proof of Lemma~\ref{lemma:coupling}.]
		As $F(A) \preccurlyeq F(B)$, enlarging the probability space if necessary, there are random configurations $A'$ and $B'$ such that $A'$ has the same distribution of $F(A)$, $B'$ has the same distribution of $F(B)$ and $A' \subseteq B'$. We then use the same open bonds to evolve the processes $(N_n^{A'})_{n \ge 0}$ and $(N_n^{B'}) _{n \ge 0}$. By attractiveness, $N_n^{A'} \subseteq N_n^{B'}$ for all $n \ge 0$. In particular, $r(N_n^{A'}) \le r(N_n^{B'})$ for all $n \ge 0$. Shifting  $(N_n^{A'})_{n \ge 0}$ by $r(A')$ and $(N_n^{B'}) _{n \ge 0}$ by $r(B')$ we obtain the desired coupling.
	\end{proof}

\section*{Acknowledgments}
The author is deeply grateful to Enrique Andjel for raising the problem and proposing an argument for its solution, as well as to Eulália Vares for her valuable discussions and careful reading of the manuscript.

\section*{Funding}
The author was supported by Conselho Nacional de Desenvolvimento Científico e Tecnológico (CNPq).

\section*{Acknowledgments}
The author is deeply grateful to Enrique Andjel for raising the problem and proposing an argument for its solution, as well as to Eulália Vares for her valuable discussions and careful reading of the manuscript.

\section*{Funding}
The author was supported by Conselho Nacional de Desenvolvimento Científico e Tecnológico (CNPq).
	
	\bibliographystyle{plainnat} % Style BST file
	\bibliography{bibliography.bib}       % Bibliography file (usually '*.bib')
	
\end{document}